\documentclass[11pt]{article}  

\usepackage{a4wide}

\usepackage{graphicx} 
\usepackage{amssymb}
\usepackage{amsmath}
\usepackage{amsthm}
\usepackage{amsfonts}
\usepackage{color}

\usepackage[latin1]{inputenc}

\newtheorem{theorem}{Theorem}[section]
\newtheorem{proposition}[theorem]{Proposition}
\newtheorem{corollary}[theorem]{Corollary}

\newtheorem{remark}[theorem]{Remark}

\newcommand{\Pmk}{\mathcal{P}^-_{k}}
\newcommand{\Ppk}{\mathcal{P}^+_{k}}

\newcommand{\Ppo}{\mathcal{P}^+_{1}}
\newcommand{\Pmo}{\mathcal{P}^-_{1}}
\newcommand{\R}{\mathbb R}
\newcommand{\RN}{\mathbb R^N}
\newcommand{\SN}{\mathcal{S}^N}

\newcommand{\tr}{{\rm Tr}}

\title{{The Dirichlet problem for fully nonlinear degenerate elliptic equations with a singular nonlinearity}}
\author{Isabeau Birindelli and Giulio Galise
}
\date{}
\begin{document}

\maketitle

\begin{abstract}
\noindent
We investigate the homogeneous Dirichlet problem in uniformly convex domains for a large class of  degenerate elliptic equations with singular zero order term. In particular we establish sharp existence and uniqueness results of positive viscosity solutions.   
\end{abstract}

\vspace{0.5cm}

{\small

  }

\section{Introduction}
In this article we investigate the existence/nonexistence of positive viscosity solutions of the singular boundary value problem
\begin{equation}\label{pb}
\left\{\begin{array}{cl}
F(D^2u)+p(x)u^{-\gamma}=0 & \text{in $\Omega$}\\
u=0 & \text{on $\partial\Omega$},
\end{array}\right.
\end{equation} 
where $\gamma>0$ and  the domain $\Omega\subset\mathbb R^N$ is  bounded and uniformly convex. 

Such kind of problem when $F$ is linear or quasilinear has been widely studied since the seminal works \cite{CRT,LM}. The survey  \cite{HM} is a good reference where an extensive literature on this subject is available. On the other hand less is known in the fully nonlinear setting. In \cite{FQS} the authors extend the  existence and regularity results of solutions as in the semilinear case to Hamilton-Jacobi-Bellman and Isaacs uniformly elliptic operators. As far as we know there are no works dealing with pure degenerate elliptic equations of fully nonlinear type. Our aim is to study \eqref{pb} in the following quite general framework:
\smallskip

\noindent the mapping $F:{\mathcal S}^N\mapsto \mathbb R$ is  continuous in $\SN$, the linear space of symmetric $N\times N$ real matrices, and  degenerate elliptic, i.e.
\begin{equation}\label{H1}\tag{H1}
F(X+Y)\geq F(X)\quad\forall X,Y\in{\mathcal S}^N,\,Y\geq0,
\end{equation}
and there exists an integer $k\in[1,N]$ such
\begin{equation}\label{H2}\tag{H2}
\Pmk(X)\leq F(X)\leq\Ppk(X)\quad\forall X\in{\mathcal S}^N.
\end{equation}
The operators $\mathcal{P}^\pm_k$ are respectively defined by the lower and upper partial sums
\begin{equation}\label{Pk}
\Pmk(X)=\sum_{i=1}^k\lambda_i(X),\quad\Ppk(X)=\sum_{i=1}^k\lambda_{N-k+i}(X)
\end{equation}
of the ordered eigenvalues $\lambda_1(X)\leq\cdots\leq \lambda_N(X)$ of $X\in\SN$. Let us mention that these extremal operators have recently generated some interest, starting with the works of Harvey and Lawson e.g.  \cite{HL1,HL2}. See also \cite{CLN,CDLV,BGL,G}.

Since we are mainly interested in degenerate equations and the results we shall presents are new when $k$ is strictly less than the dimension $N$ we assume from now on $k<N$.\\ 
The function $p:\Omega\mapsto\mathbb R$ is continuous and  satisfies for $\alpha\geq\beta$ the growth assumption
\begin{equation}\label{H3}\tag{H3}
c_1\,\delta_\Omega(x)^\alpha\leq p(x)\leq c_2\,\delta_\Omega(x)^\beta
\end{equation}
where $\delta_\Omega(x)={\rm dist}(x,\partial\Omega)$   and  $c_1,\,c_2$ are positive constants.

\medskip

Here is our existence and uniqueness result. 

\begin{theorem}\label{th1}
Assume \eqref{H1}-\eqref{H2}-\eqref{H3}, $\beta>-1$ and $\Omega$ uniformly convex. Then for any $\gamma>0$ there exists a unique $u\in C(\overline\Omega)$ positive viscosity solution of \eqref{pb}.
\end{theorem}
\noindent
By \lq\lq uniformly convex\rq\rq\ we mean that there exists $R>0$, $Y\subset\RN$ such that
\begin{equation}\label{uniformconvexity}
\Omega=\bigcap_{y\in Y}B_R(y).
\end{equation}
As usual $B_R(y)$ stands for the ball centered at $y\in\RN$ with radius $R$. When $\Omega$ is $C^2$ this is equivalent to require that all principal curvatures of the boundary are uniformly bounded from below by a positive constant, see \cite[Proposition 2.7]{BGI1}.

\medskip

The restriction on $\beta$ in Theorem \ref{th1} is sharp within the class of operator satisfying \eqref{H2}. In fact

\begin{theorem}\label{th2}
For any $\beta\leq-1$ and any $\gamma>0$ the equation
\begin{equation}\label{pb3}
\Ppk(D^2u)+\delta_{B_1}(x)^\beta u^{-\gamma}=0 \quad \text{in $B_1$}
\end{equation} 
does not admit viscosity supersolutions.
\end{theorem}

Let us try to explain the main difficulties that arise  outside the uniformly elliptic framework in order to obtain existence of solutions for \eqref{pb}. Since the classical work of Lazer and McKenna \cite{LM},  the approach typically used consists in manipulating the principal eigenfunctions of $F$  to get barrier functions, so applying the method of sub and supersolutions. In our case the first obstruction in following this approach concerns the minimal operator $\Pmk$. In \cite{BGI1} it has been proved that the the operator $\Pmk(D^2\cdot)+\mu\cdot$ satisfies the maximum principle independently on $\mu$, so preventing the existence of positive eigenfunctions. However an inspection of the proofs given in \cite{LM,FQS} shows that taking advantage of the presence of $u^{-\gamma}$ in the equation, the only property of the eigenfunction needed to construct a subsolution is that eigenfunctions look like the distance function $\delta_\Omega$ near the boundary. In view of this we shall employ a regularized version of $\delta_\Omega$ to provide a subsolution of \eqref{pb} null on $\partial\Omega$. Moreover without requiring regularity of $\partial\Omega$.\\ 
As far as the existence of a supersolution is concerned, let us emphasize that the zero order term is now competitive with $F$ and so to gain some  \lq\lq negativity\rq\rq\ the principal part  have to absorb the term $u^{-\gamma}$. When $F$ is  uniformly elliptic operators this is achieved by using functions of principal eigenfunctions. To succeed in the case of $\Ppk$, some extra assumption on $\Omega$ are needed due to the strong degeneracy of the operator. In particular we request $\Omega$ to be uniformly convex.


Let us emphasize that  Theorem \ref{th1} holds for a large class of degenerate elliptic operators, as it is shown in Section \ref{Sec2}. Also the class of domains we consider does not require any regularity, including for instance domains with corners. Finally we present a proof   without using the standard regularization $(u+\frac{1}{n})^{-\gamma}$ with $n\to\infty$. 

 The result of Theorem \ref{th2} is  different with respect to the case of the Laplacian. In fact when $k=N$, i.e. ${\mathcal P}^+_N\equiv\Delta$, the existence and uniqueness of solutions for $\Delta u+\delta_{B_1}(x)^\beta u^{-\gamma}=0$ in $\Omega$, $u=0$ on $\partial\Omega$,  still holds for $\beta\leq-1$. See \cite[Section 4-v]{LM}. 

In Section 2 we collect some examples of degenerate operators satisfying \eqref{H1}-\eqref{H2} and some preliminary results used in the rest of the article. Section 3 is devoted to the proofs of Theorems \ref{th1}-\ref{th2}. In Section 4 we generalize the above theorems including first order terms showing that also in this case new nonexistence phenomena occur.

\section{Examples and preliminary results}\label{Sec2}

The class of operators satisfying conditions \eqref{H2} is quite large and contains important examples, we include a few of them here. 
\subsection{Examples}

\noindent
\textbf{1. Linear operators with $k$-directions of uniformly ellipticity}\\
Equations of the form
\begin{equation*}
\sum_{i=1}^k\frac{\partial^2 u}{\partial x_{j_i}^2}+p(x)u^{-\gamma}=0\quad\text{in $\Omega$}
\end{equation*}
fits into our framework. Here $1\leq j_1<j_2<\ldots<j_k\leq N$ are integer numbers. 
The corresponding operators $F:{\mathcal S}^N\mapsto\mathbb R$  are 
\begin{equation*}
F(X)=\sum_{i=1}^k X e_{j_i}\cdot e_{j_i}
\end{equation*}
where $\left\{e_1,\ldots,e_N\right\}$ is the standard basis of $\mathbb R^N$. We have
\begin{equation*}
\begin{split}
F(X)&\leq\sup\left\{\sum_{i=1}^k X v_i\cdot v_i\;\;\text{s.t.}\;\;\text{$v_i\in\RN$ and $v_i\cdot v_j=\delta_{ij}$}\right\}=\Ppk(X)\\
F(X)&\geq\inf\left\{\sum_{i=1}^k X v_i\cdot v_i\;\;\text{s.t.}\;\;\text{$v_i\in\RN$ and $v_i\cdot v_j=\delta_{ij}$}\right\}=\Pmk(X).
\end{split}
\end{equation*} 
More general we can deal with  
$$
\tr(AD^2u)+p(x)u^{-\gamma}=0\quad\text{in $\Omega$}
$$
where $A\in\SN$ is a projection matrix on a $k$-dimensional subspace of $\RN$.

\bigskip
\noindent
\textbf{2. Functions of eigenvalues}\\
Partial sums of eigenvalues 
$$
F(X)=\sum_{i=1}^k\lambda_{j_i}(X),
$$
including the extremal case:
\begin{enumerate}
	\item $k=N$, $F(D^2u)=\Delta u$ 
	\item $k=1$, $F(D^2u)=\lambda_j(D^2u)$.
\end{enumerate}

\noindent As further example of the generality  we are concerning with we can also consider Bellmann/Isaacs type operators such as
$$
F(X)=\sup_{\alpha}\inf_{\beta} F_{\alpha,\beta}(X)
$$
where the parameters $\alpha,\beta$ lies in some sets on index $\mathcal A, \mathcal B$ and the operators $F_{\alpha,\beta}$ satisfy \eqref{H1}-\eqref{H2}.

In fact the class of operators for which Theorem \ref{th1} holds is much larger. It includes operators that may depends also on $x, u, Du$ as long they are bounded from above by $\Ppk$ and from below by $\Pmk$ and for which comparison principle holds. For example

\bigskip
\noindent
\textbf{3. Infinity Laplacian}\\
Consider the 1-homogeneous infinity  Laplacian
\begin{equation*}
\begin{split}
\Delta_\infty u&=\frac{1}{|Du|^2}\,D^2uDu\cdot Du\\
&=\frac{1}{|Du|^2}\tr\left(D^2u Du\otimes Du\right).
\end{split}
\end{equation*}
The nonlinearity $F(q,X)=\frac{1}{|q|^2}\tr\left(Xq\otimes q\right)$ is degenerate elliptic on the set $\SN\times\RN\backslash\left\{0\right\}$ and undefined at $q=0$. Following \cite[Section 9]{CIL} we have to use the lower and upper semicontinuous extension of $F$ to $(0,X)$ given by
$$
\underline F(q,X)=\left\{
\begin{array}{cl}
F(q,X) & \text{if $q\neq0$}\\
\lambda_1(X) & \text{if $q=0$}
\end{array}
\right.
\quad\text{and}\quad
\overline F(q,X)=\left\{
\begin{array}{cl}
F(q,X) & \text{if $q\neq0$}\\
\lambda_N(X) & \text{if $q=0$}
\end{array}
\right.
$$
in such a way comparison principle applies. Condition \eqref{H2} is satisfied with $k=1$, in the sense that 
$$
\Pmo(X)\leq\underline F(q,X)\leq\overline F(q,X)\leq\Ppo(X)\,.
$$

\medskip
\noindent
See also Remark \ref{further example} for further generalizations.

\subsection{Preliminary results}

Since the dependence in $u$ in the equation $F(D^2u)+p(x)u^{-\gamma}=0$ is monotone decreasing, then the standard arguments for comparison principle of \cite{CIL} applies. Hence we have the following

\begin{theorem}[\textbf{Comparison principle}]\label{CP}
Assume \eqref{H1}-\eqref{H2} and $p(x)$ positive in $\Omega$. If $u,\,v$ are respectively viscosity sub and supersolution of 
$$F(D^2u)+p(x)u^{-\gamma}=0 \quad \text{in $\Omega$}$$
and $\limsup_{x\to\partial\Omega}(u-v)(x)\leq0$, then $u\leq v$ in $\Omega$.
\end{theorem}

Next Theorem is extracted from \cite[Chapter VI, \S 2]{S}

\begin{theorem}[\textbf{Regularized distance}]\label{dist}
Let $\Omega\subset\RN$ be a bounded domain. There exists a function $d(x)=d(x,\partial\Omega)$ such that for any $x\in\Omega$ 
\begin{itemize}
	\item[a)] $\displaystyle C_1\delta_\Omega(x)\leq d(x)\leq C_2\delta_\Omega(x)$
	\item[b)] $d\in C^\infty(\Omega)$ and for any multiindex $\alpha$
	$$
	\left|\frac{\partial^\alpha}{\partial x^\alpha}d(x)\right|\leq B_\alpha \delta_\Omega(x)^{1-|\alpha|}
	$$
	where $C_1,\,C_2,\,B_\alpha$ are independent on $\partial\Omega$.
\end{itemize}
 
\end{theorem}

\section{Proofs}
%
%
%
%
%

\begin{proof}[Proof of Theorem \ref{th1}] Thanks to Theorem \ref{CP} we are in a position to  use the Perron method. For its application we are going to construct continuous sub and a supersolution  of \eqref{pb} vanishing on the boundary of $\Omega$.

\medskip

\noindent
Let $d(x)$ be the regularized distance function from $\partial\Omega$, see Theorem \ref{dist}. Let
$$
\underline u(x)=\varepsilon d(x)^t
$$
where  $t=\frac{\alpha+2}{\gamma+1}$ and $\varepsilon$ is a small positive number to be determined. We have
$$
D^2\underline u(x)=\varepsilon t d(x)^{t-2}\left((t-1)Dd(x)\otimes Dd(x)+d(x)D^2d(x)\right)
$$
and, using the properties of $d$, there exists $C>0$ big enough such that
\begin{equation*}
\begin{split}
F(D^2\underline u(x))&+p(x)\underline u(x)^{-\gamma} \geq\Pmk (D^2\underline u(x))+c_1\delta(x)^\alpha\underline u(x)^{-\gamma}\\
&\geq\varepsilon t d(x)^{t-2}\left(\Pmk\left((t-1)Dd(x)\otimes Dd(x)\right)+\Pmk\left(d(x)D^2d(x)\right)\right)+c_1\delta(x)^\alpha\underline u(x)^{-\gamma}\\
&\geq-\varepsilon t d(x)^{t-2}|t-1||Dd(x)|^2+\varepsilon t d(x)^{t-1}\Pmk(D^2d(x))+c_1\varepsilon^{-\gamma}\delta(x)^\alpha d(x)^{-\gamma t}\\
&\geq \frac{\varepsilon d(x)^{t-2}}{C}\left(\frac{1}{\varepsilon^{\gamma+1}}-C^2t(1+|t-1|)\right).
\end{split}
\end{equation*}
Hence taking $\varepsilon$ small enough  
$$
F(D^2\underline u(x))+p(x)\underline u(x)^{-\gamma} \geq0\quad\text{in $\Omega$}
$$
and $\underline u=0$ on $\partial\Omega$.

\medskip
\noindent
Now we are going to construct a continuous supersolution $\overline u$ of \eqref{pb} such that $\overline u=0$ on $\partial\Omega$. For this we first look at the auxiliary problem 
\begin{equation}\label{pb2}
\left\{\begin{array}{cl}
\Ppk(D^2u)+c_2(R-|x-y|)^\beta u^{-\gamma}=0 & \text{in $B_R(y)$}\\
u=0 & \text{on $\partial B_R(y)$}
\end{array}\right.
\end{equation} 
for $\beta\in(-1,0]$. Note that $R-|x-y|=\delta_{B_R(y)}(x)$.\\
By a straightforward computation the function
\begin{equation}\label{radial solution}
u(r)=\left(\frac{c_2(1+\gamma)}{k}\left(\frac{R}{\beta+1}(R-r)^{\beta+1}-\frac{1}{\beta+2}(R-r)^{\beta+2}\right)\right)^{\frac{1}{1+\gamma}}
\end{equation}
is the solution of the ODE problem
\begin{equation*}
\left\{\begin{array}{cl}
k\frac{u'(r)}{r}+c_2(R-r)^\beta u(r)^{-\gamma}=0 & \text{for $r\in(0,R)$}\\
u'(0)=0\\
 u(R)=0.  
\end{array}\right.
\end{equation*}
Since $\beta\leq0$, $u'\leq0$ and $u>0$ in $[0,R)$ then
\begin{equation*}
\begin{split}
u''(r)&=\frac{u'(r)}{r}+\frac{c_2}{k}r(R-r)^{\beta-1}u(r)^{-\gamma-1}\left(\beta u(r)+\gamma(R-r)u'(r)\right)\\
&\leq \frac{u'(r)}{r}\,.
\end{split}
\end{equation*}
Hence, if $r=|x-y|$, the function $u$ defined by \eqref{radial solution} is the unique solution of \eqref{pb2}.\\
If $\beta>0$ then the function
\begin{equation}\label{radial supersolution}
u(r)=\left(\frac{c_2 R^\beta(1+\gamma)}{2k}(R^2-r^2)\right)^{\frac{1}{1+\gamma}}
\end{equation}
 is  a supersolution of \eqref{pb2}. This trivially follows from the inequality $(R-r)^\beta\leq R^\beta$ and the fact that \eqref{radial supersolution} is the solution in $B_R(y)$ of $\Ppk(D^2u)+c_2R^\beta u^{-\gamma}=0$, $u=0$ on $\partial B_R(y)$.\\
In this way for any $\beta>-1$ we have found a continuous supersolution of \eqref{pb2} vanishing on the boundary. Moreover if $\sigma=\min\left\{\frac{1}{\gamma+1},\frac{\beta+1}{\gamma+1}\right\}$ there exists a constant $C=C(k,\gamma,\beta,R,c_2)$ such that
\begin{equation}\label{holder}
|u(r_1)-u(r_2)|\leq C|r_1-r_2|^\sigma
\end{equation}
for any $r_1,\,r_2\in[0,R]$.\\
Now we use the uniformly convexity of  $\Omega$, i.e.
\begin{equation}
\Omega=\bigcap_{y\in Y}B_R(y),
\end{equation}
to provide $\overline u$. For any $y\in Y$ and $x\in B_R(y)$ let us denote by $u_y(r)$, $r=|x-y|$, the supersolution of \eqref{pb2} constructed above. Define
\begin{equation}\label{supersolution}
\overline u(x)=\inf_{y\in Y}u_y(r).
\end{equation}
We claim that $\overline u$ yields a continuous supersolution of \eqref{pb}, positive in $\Omega$ and null on $\partial\Omega$.\\
For any $x_1,\,x_2\in\overline\Omega$, using \eqref{holder}
\begin{equation*}
\begin{split}
|\overline u(x_1)-\overline u(x_2)|&\leq\sup_{y\in Y}|u_y(|x_1-y|)-u_y(|x_2-y|)|\\
&\leq C\left||x_1-y|-|x_2-y|\right|^\sigma\leq C|x_1-x_2|^\sigma.
\end{split}
\end{equation*}
Hence $\overline u$ is H$\rm\ddot{o}$lder continuous in $\overline\Omega$.\\
We now show that $\overline u$ is positive in $\Omega$. Let $x_0\in\Omega$, in this way $\delta_\Omega(x_0)>0$.  Moreover, since $\Omega\subseteq B_R(y)$, for any $y\in Y$ it holds that $\delta_{B_R(y)}(x_0)\geq \delta_\Omega(x_0)$. If $\beta\in(-1,0]$ 
\begin{equation}\label{lb1}
\begin{split}
u_y(|x_0-y|)&=\left(\frac{c_2(1+\gamma)}{k}\left(\frac{R}{\beta+1}(\delta_{B_R(y)}(x_0))^{\beta+1}-\frac{1}{\beta+2}(\delta_{B_R(y)}(x_0))^{\beta+2}\right)\right)^{\frac{1}{1+\gamma}}\\
&\geq \left(\frac{c_2(1+\gamma)}{k}\left(\frac{R}{\beta+1}(\delta_\Omega(x_0))^{\beta+1}-\frac{1}{\beta+2}(\delta_\Omega(x_0))^{\beta+2}\right)\right)^{\frac{1}{1+\gamma}}
\end{split}
\end{equation} 
where we have used the monotonicity of the map $t\in [0,R]\mapsto \frac{R}{\beta+1}t^{\beta+1}-\frac{1}{\beta+2}t^{\beta+2}$.\\
If $\beta>0$ it holds
\begin{equation}\label{lb2}
\begin{split}
u_y(|x_0-y|)&\geq\left(\frac{c_2 R^{\beta+1}(1+\gamma)}{2k}\delta_{B_R(y)}(x_0)\right)^{\frac{1}{1+\gamma}}\\
&\geq \left(\frac{c_2 R^{\beta+1}(1+\gamma)}{2k}\delta_\Omega(x_0)\right)^{\frac{1}{1+\gamma}}.
\end{split}
\end{equation}
Since the lower bounds in \eqref{lb1}-\eqref{lb2} are positive and independent on $y\in Y$ then $\overline u$ is strictly positive in $\Omega$.\\
As far as the boundary Dirichlet condition is concerned fix any $x_0\in\partial\Omega$. Then there exists $y_{x_0}\in Y$ such that $x_0\in\partial B_R(y_{x_0})$ and 
$$
\overline u(x_0)\leq u_{y_0}(x_0)=0.
$$
Since $x_0$ is arbitrary we conclude that $\overline u=0$ on $\partial\Omega$.

 It remains to prove  that $\overline u$ is supersolution. By standard argument it is sufficient to show that for any $y\in Y$ the function $u_y$ is supersolution of 
$$
\Ppk(D^2u)+c_2\delta_\Omega(x)^\beta u^{-\gamma}=0\qquad\text{in $\Omega$}.
$$
If $\beta\geq0$ this is immediate, since by construction $u_y$ is  solution of
$$
\Ppk(D^2u)+c_2\delta_{B_R(y)}(x)^\beta u^{-\gamma}=0\qquad\text{in $B_r(y)$}
$$ 
and $\delta_{B_R(y)}(x)\geq\delta_\Omega(x)$.\\
Now we consider the case $\beta\in(-1,0)$. Let $x_0\in\Omega$ and let $\varphi\in C^2(\Omega)$ such that 
$$
(\overline u-\varphi)(x)\geq (\overline u-\varphi)(x_0)=0\quad\forall x\in\Omega.
$$
Select $y_0\in Y$, depending on $x_0$, such that 
\begin{equation}\label{point}
\delta_\Omega(x_0)=\delta_{B_R(y_0)}(x_0).
\end{equation}
 We claim that 
\begin{equation}\label{claim}
\overline u(x_0)=u_{y_0}(|x_0-y_0|).
\end{equation}
As in \eqref{lb1} for any $y\in Y$
\begin{equation*}
\begin{split}
u_y(|x_0-y|)&=\left(\frac{c_2(1+\gamma)}{k}\left(\frac{R}{\beta+1}(\delta_{B_R(y)}(x_0))^{\beta+1}-\frac{1}{\beta+2}(\delta_{B_R(y)}(x_0))^{\beta+2}\right)\right)^{\frac{1}{1+\gamma}}\\
&\geq \left(\frac{c_2(1+\gamma)}{k}\left(\frac{R}{\beta+1}(\delta_{B_R(y_0)}(x_0))^{\beta+1}-\frac{1}{\beta+2}(\delta_{B_R(y_0)}(x_0))^{\beta+2}\right)\right)^{\frac{1}{1+\gamma}}\\
&=u_{y_0}(|x_0-y_0|)
\end{split}
\end{equation*}
where we have used that fact that $\delta_{B_R(y)}(x_0)\geq \delta_\Omega(x_0)=\delta_{B_R(y_0)}(x_0)$. This implies \eqref{claim}. \\
Hence $\varphi$ is a test function touching from below $u_{y_0}$ at $x_0$. Using \eqref{point} we conclude
$$
\Ppk(D^2\varphi(x_0))+c_2\delta_\Omega(x_0)^\beta\varphi(x_0)^{-\gamma}=\Ppk(D^2\varphi(x_0))+c_2\delta_{B_R(y_0)}(x_0)^\beta\varphi(x_0)^{-\gamma}\leq0.
$$
\end{proof}

From the proof of Theorem \ref{th1} we immediately obtain the following  

\begin{corollary}
Let $u$ be the solution of \eqref{pb} provided by Theorem \ref{th1}. Then there exists positive constants $a_i=a_i(k,R,c_1,c_2,\alpha,\beta,\gamma)$ for $i=1,2$ such that 
\begin{equation}\label{estimate}
a_1\delta_\Omega(x)^{\frac{\alpha+2}{\gamma+1}}\leq u(x)\leq a_2{\delta_\Omega(x)}^\sigma,\quad x\in\overline\Omega,
\end{equation}
where $\sigma=\min\left\{\frac{1}{\gamma+1},\frac{\beta+1}{\gamma+1}\right\}$.
\end{corollary}

\begin{remark}
\rm Note that in the cases $F=\Ppk$ and $\beta\leq0$   the solution of \eqref{pb} in the ball $B_R(y)$ is explicit, see \eqref{radial solution}. From this we  obtain a  more precise information of the solution $u$ near the boundary of $\Omega$ and  \eqref{estimate} reduces to
\begin{equation}\label{estimate2}
a_1\delta_\Omega(x)^\sigma\leq u(x)\leq a_2{\delta_\Omega(x)}^\sigma,\quad x\in\overline\Omega.
\end{equation}
It is remarkable that even in the simplest case $\beta=0$, the best regularity we can expect is then $C^{0,\frac{1}{1+\gamma}}(\overline\Omega)$. Accordingly $|Du(x)|\to\infty$ as $x\to\partial\Omega$, independently on $\gamma>0$. This is a main difference with respect to uniformly elliptic setting, see \cite[Theorem 1.2]{LM}, \cite[Theorem 1.2]{DP} and \cite[Theorems 2 and 8]{FQS}, where the gradient stay bounded in $\Omega$ depending on whether $\gamma<1$ or $\gamma>1$.  
\end{remark}

\begin{remark}\label{further example}
\rm It is worth to point out that the proof of Theorem \ref{th1}, in particular the construction of the supersolution defined by \eqref{supersolution}, still works for some degenerate elliptic operators which don't satisfy \eqref{H2} for any $X\in\SN$, but they do only in some proper subset of $\SN$. For instance let us the consider  the elliptic operator 
\begin{equation*}
F(q,X)=\tr\left(I-\frac{q\otimes q}{1+|q|^2}\right)X.
\end{equation*}
Note that $F(Du,D^2u)=0$ is the equation of minimal surfaces in nonparametric form. \\
If  we restrict the domain of $F$ to  $q\in\RN$ and $X\in{\mathcal S}_-=\left\{X\in\SN\;\text{s.t.}\;\lambda_1(X)\leq0\right\}$ we have
\begin{equation*}
F(q,X)\leq \tr X-\frac{|q|^2}{1+|q|^2}\lambda_1(X)\leq{\mathcal P}^+_{N-1}(X).
\end{equation*}
Then the function \eqref{supersolution},  which is concave, is in turn supersolution of the equation
\begin{equation}\label{minsurf}
\Delta u+\frac{D^2uDu\cdot Du}{1+|Du|^2}+p(x)u^{-\gamma}=0 \quad \text{in $\Omega$}.
\end{equation}
Concerning the construction of a subsolution vanishing on $\partial\Omega$ we can argue as in the above proof. Hence we obtain  existence and uniqueness for \eqref{minsurf} with $u=0$ on $\partial\Omega$. 
\end{remark}

\begin{proof}[Proof of Theorem \ref{th2}]
For any $\rho\in(0,1)$ let us consider the ODE problem
\begin{equation*}
\left\{\begin{array}{cl}
k\frac{w'(r)}{r}+\frac{w(r)^{-\gamma}}{1-r} =0 & \text{for $r\in(0,\rho)$}\\
w'(0)=0\\
 w(\rho)=0.  
\end{array}\right.
\end{equation*}
By computations
$$
w(r)=\left(\frac{1+\gamma}{k}\left(r-\rho+\log\frac{1-r}{1-\rho}\right)\right)^{\frac{1}{1+\gamma}}
$$
and 
$$
w''=\frac{w'}{r}-\frac{r}{k(1-r)^2}w^{-\gamma}+\frac{\gamma r}{k(1-r)}w^{-\gamma-1}w'\leq\frac{w'}{r}\,.
$$
Hence $w(|x|)$ is the solution of
\begin{equation}\label{pbrho}
\left\{\begin{array}{cl}
\Ppk(D^2w)+\delta_{B_1}(x)^{-1}w^{-\gamma}=0 & \text{in $B_\rho$}\\
w=0 & \text{on $\partial B_\rho$}.
\end{array}\right.
\end{equation}
Let us assume by contradiction that there exists a supersolution $u$ of \eqref{pb3}. Since $\beta\leq-1$ and $\delta_{B_1}\leq1$ then $u$ is in turn a positive supersolution of \eqref{pbrho}. The comparison principle yields  $u(x)\geq w(|x|)$ in $B_\rho$. Hence for any $x\in B_1$
$$
u(x)\geq\lim_{\rho\to1}w(|x|)=\infty
$$
contradiction. 
\end{proof}

We conclude this section by few considerations about solutions of
\begin{equation}\label{eq20}
\left\{\begin{array}{cl}
F(D^2u)+u^{-\gamma}=0 & \text{in $B_1$}\\
u=0 & \text{on $\partial B_1$}
\end{array}\right.
\end{equation}
for some explicit radial invariant operators of interests for this paper, namely  truncated Laplacians ${\mathcal P}^\pm_k$ and the infinity Laplacian, and their connection with the full Laplacian $\Delta$. \\
In the proof of Theorem \ref{th1} we found that $u(r)=\left(\frac{1+\gamma}{2k}(1-r^2)\right)^{\frac{1}{1+\gamma}}$ is the radial solution of \eqref{eq20} for $F=\Ppk$. 

In order to solve \eqref{eq20} for  $F=\Pmk$, we study the second order problem
\begin{equation*}\label{eq21}
\left\{\begin{array}{cl}
u''(r)+(k-1)\frac{u'(r)}{r}+u(r)^{-\gamma}=0 & \text{for $r\in(0,\rho)$}\\
u(\rho)=u'(0)=0.  
\end{array}\right.
\end{equation*}
Note that $r^{k-1}u'(r)$ is decreasing, hence $u$ is decreasing and positive in $[0,\rho)$. Moreover 
$$
\left(u''-\frac{u'}{r}\right)'=-\frac kr\left(u''-\frac{u'}{r}\right)+\gamma u^{-\gamma-1}u'\leq -\frac kr\left(u''-\frac{u'}{r}\right)
$$
i.e.
$\left(r^{k}\left(u''-\frac{u'}{r}\right)\right)'\leq 0$,
from which we infer that $u''\leq\frac{u'}{r}$ for any $r\in[0,\rho)$ and that $u(|x|)$ is solution of \eqref{eq20} in $B_\rho$. 
Using the scaling invariance of the problem, $v(x):=\alpha^{-\frac{2}{\gamma+1}}u(\alpha |x|)$ is solution if $u$ is, hence we can pick $\alpha$ such that $\rho=1$. Observe that $\Pmk$ acts in this setting similarly to the Laplacian in dimension $k$. 

The case $k=1$ leads us to the solution not only of $\Pmo$, but also of $\Delta_\infty$. For this set $U(x)=u(r)$. For $x\neq0$ then $|DU(x)|=|u'(r)|\neq0$ and so $\Delta_\infty U(x)=u''(r)=U(x)^{-\gamma}$. If $x=0$ one has $D^2U(0)=u''(0)I$, hence $\lambda_1(D^2U(0))=\lambda_N(D^2U(0)=-U(0)^{-\gamma}$.  
Let us point put that $U\notin C^1(\overline B_1)$ if and only if $\gamma\geq1$. This follows from the fact that for $r\in[0,1)$ it holds that  $\frac{u'(r)^2}{2}+\frac{u(r)^{1-\gamma}}{1-\gamma}=\frac{u(0)^{1-\gamma}}{1-\gamma}$ if $\gamma\neq1$ and $\frac{u'(r)^2}{2}+\log(u(r))=\log(u(0))$ if $\gamma=1$.

\smallskip

For the sake of completeness let us mention that there are also  cases in which the solution of \eqref{pb} in $B_1$ with $p(x)=\delta_{B_1}(x)^\alpha$, $\alpha>0$, is explicit as well. Consider the minimal operator $F=\Pmk$, or more general any partial sum of eigenvalues of the form $F(D^2u)=\sum_{i=1}^k\lambda_{j_i}(D^2u)$ with $k<N$  and $j_k<N$. By a straightforward computation the radial function
$$
u(|x|)=\left(\frac{1+\gamma}{k(1+\alpha)}(1-|x|)^{\alpha+1}\left(|x|+\frac{1}{\alpha+2}(1-|x|)\right)\right)^{\frac{1}{1+\gamma}}
$$ 
is solution of $F(D^2u)+\delta_{B_1}(x)^\alpha u^{-\gamma}=0$ in $B_1$, $u=0$ on $\partial B_1$ as long as $\alpha\geq\gamma$. Moreover 
$u\in C^{1}(\overline B_1)$.

\section{Generalizations}
The proofs of Theorems \ref{th1}-\ref{th2} extend to some cases of equations depending also on first order terms, such as 

\begin{equation}\label{pbH}
\left\{\begin{array}{cl}
F(D^2u)+H(x,Du)+p(x)u^{-\gamma}=0 & \text{in $\Omega$}\\
u=0 & \text{on $\partial\Omega$}
\end{array}\right.
\end{equation} 
where $H\in C(\Omega\times\RN)$ satisfies the structure conditions: $\exists b\in\R_+$ such that
\begin{equation}\label{H4}\tag{H4}
\left|H(x,q)\right|\leq b|q|\quad\forall (x,q)\in\Omega\times\RN
\end{equation}
and there exists $\omega$ a modulus of continuity such that 
\begin{equation}\label{H5}\tag{H5}
\left|H(x,q)-H(y,q)\right|\leq \omega(|x-y|(1+|q|))\quad\forall (x,y,q)\in\Omega^2\times\RN.
\end{equation}
Assumption \eqref{H5} is standard for the validity of comparison principle.

\medskip

In this section we are concerned with the existence of solutions within this large class of equation. We shall see that the existence results are very sensitive to the \lq\lq size of $b$\rq\rq\ in \eqref{H4}, even in in the simplest case  $\Omega=B_R$ and $H(x,Du)=b|Du|$   problem \eqref{pbH} has no solutions (in fact supersolutions) if $b$ is too large with respect $R$. This is the case for instance of the partial sums (see example 2, Section \ref{Sec2}) of the form 
\begin{equation}\label{4eq1}
F(D^2u)=\sum_{i=1}^k\lambda_{j_i}(D^2u), \quad k<N\quad\text{and}\quad j_1>1.
\end{equation}

\begin{proposition}\label{noexistenceH}
Let $F$ as in \eqref{4eq1}. If $bR\geq k$ then there are no positive viscosity supersolutions of 
\begin{equation}\label{4eq2}
F(D^2u)+b|Du|+u^{-\gamma}=0 \quad\text{in $B_R$.}
\end{equation}
\end{proposition}
\begin{proof}
For $\varepsilon>0$ the  function
$$
u_\varepsilon(r)=\left(\frac{1+\gamma}{b}\left(r-\frac kb+\varepsilon+\frac kb\log\frac{k-br}{\varepsilon b}\right)\right)^\frac{1}{1+\gamma}
$$
is solution of 
\begin{equation*}
\left\{
\begin{array}{cl}
k\frac{u'(r)}{r}-bu'(r)+u(r)^{-\gamma}=0 & \text{ for $r\in(0,\frac kb-\varepsilon)$}\\
u'(0)=0,\;u(\frac kb-\varepsilon)=0. & 
\end{array}\right.
\end{equation*}
Moreover for $r<\frac kb-\varepsilon$
\begin{equation}\label{4eq3}
u_\varepsilon'(r)=-\frac{r}{k-br}u_\varepsilon(r)^{-\gamma}\leq0
\end{equation}
and 
\begin{equation}\label{4eq4}
\begin{split}
u_\varepsilon''(r)&=\frac{u_\varepsilon'(r)}{r}-\frac{r}{(k-br)^2}u_\varepsilon(r)^{-2\gamma-1}\left(\gamma r+b u_\varepsilon(r)^{\gamma+1}\right)\\
&\leq\frac{u_\varepsilon'(r)}{r}.
\end{split}
\end{equation}
In view of \eqref{4eq3}-\eqref{4eq4} we infer that $u_\varepsilon(|x|)$ solves
\begin{equation*}
\left\{\begin{array}{cl}
F(D^2u_\varepsilon)+b|Du|+u_\varepsilon^{-\gamma}=0 & \text{in $B_{\frac kb-\varepsilon}$}\\
u_\varepsilon=0 & \text{on $\partial B_{\frac kb-\varepsilon}$}.
\end{array}\right.
\end{equation*}
Let us assume by contradiction that there exists $u$ positive supersolution of \eqref{4eq2}. Since $B_{\frac kb-\varepsilon}\subset B_R$ and $u_\varepsilon=0$ on $\partial B_{\frac kb-\varepsilon}$ then comparison principle yields $u(x)\geq u_\varepsilon(|x|)$ in $B_{\frac kb-\varepsilon}$. Sending $\varepsilon\to0$ we obtain that $u\equiv\infty$ in $B_{\frac kb}$, contradiction. 
\end{proof}


\begin{remark}\rm
It is worth to point out that the nonexistence result expressed by Proposition \ref{noexistenceH} is no longer valid in general if we replace $u^{-\gamma}$ with $\delta_\Omega(x)^\alpha u^{-\gamma}$, with $\alpha>0$. Indeed by a direct computation one can show that if $bR=k$ and $\alpha\in(0,1$),  then the function
\begin{equation*}
u(|x|)=\left((1+\gamma)\int_{|x|}^R\frac{s(R-s)^\alpha}{k-bs}\,ds\right)^{\frac{1}{1+\gamma}}=\left(\frac{(1+\gamma)R}{k\alpha(\alpha+1)}(R-|x|)^\alpha\left(R+\alpha|x|\right)\right)^{\frac{1}{1+\gamma}}
\end{equation*}
is solution of $\Ppk(D^2u)+b|Du|+\delta_{B_R}(x)^\alpha u^{-\gamma}=0$ in $B_R$ and $u=0$ on $\partial B_R$. 
\end{remark}

The above Proposition shows that the condition $bR\geq k$ is a real obstruction to the existence of solutions of \eqref{pbH}. Nevertheless, as soon as $bR<k$,  the analogous Theorem \ref{th1} holds true.

\begin{theorem}\label{th1H}
Let $\Omega$ be bounded and uniformly convex domain. Assume that $F$, $p(x)$ and  $H$  satisfy respectively the assumptions \eqref{H1}-\eqref{H2},  \eqref{H3} and  \eqref{H4}-\eqref{H5}. If $bR<k$ and   $\beta>-1$,  then for any $\gamma>0$ there exists a unique $u\in C(\overline\Omega)$ positive viscosity solution of \eqref{pbH}.
\end{theorem}

\begin{proof}[Sketch of the proof]
The function $\underline u(x)=\varepsilon d(x)^{\frac{\alpha+2}{\gamma+1}}$ is still a subsolution by choosing $\varepsilon$ suitably small.\\
Concerning the existence of a supersolution of \eqref{pbH} vanishing on the boundary of $\Omega$, let us consider the problem 
\begin{equation}\label{4eq5}
\left\{\begin{array}{cl}
\Ppk(D^2u)+b|Du|+c_2\delta_{B_R(y)}^\beta u^{-\gamma}=0 & \text{in $B_R(y)$}\\
u=0 & \text{on $\partial B_R(y)$}.
\end{array}\right.
\end{equation}
Set $r=|x-y|$ for $x\in B_R(y)$. The function
$$
u_y(r)=\left(c_2(1+\gamma)\int_r^R\frac{s(R-s)^\beta}{k-bs}\,ds\right)^{\frac{1}{1+\gamma}}
$$
is solution of \eqref{4eq5} if $\beta\in(-1,0]$. Moreover $u\in C^{0,\frac{\beta+1}{\gamma+1}}(\overline B_R(y))$ uniformly with respect to $y$.\\
If instead $\beta>0$ then the function
$$
u_y(r)=\left(\frac{c_2 R^\beta(1+\gamma)}{b}\left(r-R+\frac kb\log\frac{k-br}{k-bR}\right)\right)^{\frac{1}{1+\gamma}}
$$
is supersolution of \eqref{4eq5}, $u=0$ on $\partial B_{R}(y)$ and $u\in C^{0,\frac{1}{\gamma+1}}(\overline B_R(y))$ uniformly with respect to $y$. Then according to the sign of $\beta$ and following the arguments of the proof of Theorem \ref{th1} with minor changes,  the function $$\overline u(x)=\inf_{y\in Y}u_y(|x-y|),\quad x\in\overline\Omega$$
provides the desired supersolution of \eqref{pbH}. 
\end{proof}

\begin{theorem}\label{th2H}
Let $bR\leq k$ and let $\gamma>0$. Then for any $\beta\leq-1$ the equation
\begin{equation}\label{4eq6}
\Ppk(D^2u)-b|Du|+\delta_{B_1}(x)^\beta u^{-\gamma}=0 \quad \text{in $B_R$}
\end{equation} 
has no viscosity supersolutions.
\end{theorem}
\begin{proof}
For any $\rho\in(0,R)$ the function
$$
w(r)=\left((1+\gamma)\int_{r}^\rho\frac{s(R-s)^\beta}{(k+bs)}\,ds\right)^{\frac{1}{1+\gamma}}
$$
is solution of 
\begin{equation*}
\left\{\begin{array}{cl}
k\frac{w'}{r}+bw'+(R-r)^\beta w^{-\gamma}=0 & \text{for $r\in(0,\rho)$}\\
w'(0)=0,\quad w(\varrho)=0. & 
\end{array}\right.
\end{equation*}
Using the assumptions $\beta\leq-1$ and $bR\leq k$ we have 
\begin{equation*}
\begin{split}
w''(r)&=\frac{w'}{r}-r\frac{(R-r)^{2\beta}}{(k+br)^2}w(r)^{-2\gamma-1}\left[\left(-b(R-r)-\beta(k+br)\right)(R-r)^{-\beta-1}w(r)^{\gamma+1}+r\gamma\right]\\
&\leq \frac{w'}{r}.
\end{split}
\end{equation*}
Hence we infer that $w(|x|)$ is solution 
\begin{equation*}
\left\{\begin{array}{cl}
\Ppk(D^2w)-b|Dw|+\delta_{B_1}(x)^{-\beta}w^{-\gamma}=0 & \text{in $B_\rho$}\\
w=0 & \text{on $\partial B_\rho$}.
\end{array}\right.
\end{equation*}
Since $w\to\infty$ as $\rho\to R$, then the comparison principle prevent the existence of supersolutions of \eqref{4eq6}.
\end{proof}

\bigskip
\noindent
\textsc{I. Birindelli, G. Galise}: Dipartimento di Matematica \lq\lq Guido Castelnuovo\rq\rq,\\
Sapienza Universit\`a di Roma, P.le Aldo Moro 2, I-00185 Roma, Italy.\\
E-mail: \texttt{isabeau@mat.uniroma1.it}\\
E-mail: \texttt{galise@mat.uniroma1.it}

\end{document}